\documentclass{article}
\usepackage[affil-it]{authblk}
\usepackage{etex}
\usepackage{lineno,hyperref}
\usepackage{indentfirst}
\usepackage[below]{placeins} % use \FloatBarrier in the b1ody

\usepackage{amsmath,amssymb,amsfonts,amsthm,latexsym}
\usepackage{amsfonts}
\usepackage[utf8]{inputenc} \usepackage{ae,aecompl}     % changes to Type 1 fonts
\usepackage{amscd}
\usepackage{mathrsfs}       % fonts for math script (\mathscr)
\usepackage{fancyhdr}
\usepackage{graphicx,psfrag}
\usepackage{epsf} 
\usepackage[all]{xy}
\usepackage{tikz}
\usetikzlibrary{chains,scopes,shapes,arrows,trees,matrix,positioning,fit,decorations.pathmorphing}

\usepackage[nofoot,letterpaper,left=1.75truein, top=1truein]{geometry}

\usepackage{multicol}

\theoremstyle{definition}

\newtheorem{thm}{Theorem}[section]% or [chapter] or [subsection]
\newtheorem{lem}[thm]{Lemma}

\newtheorem{cor}[thm]{Corollary}
\newtheorem{pro}[thm]{Proposition}

\theoremstyle{definition}%%%%%%%%%%%%%%%%%%%%%%%%%%%%%%%%%%

\newcommand {\C}{\mathbb C}
\newcommand {\Z}{\mathbb Z}
\newcommand {\Q}{\mathbb Q}

\newcommand {\nb}{\mathcal N}
\newcommand {\RP}{\mathbb R \mathbb P}

\begin{document}

\title{Exceptional cosmetic surgeries on homology spheres}
\author{Huygens C. Ravelomanana}
\maketitle
\date{}

\begin{abstract}
The cosmetic surgery conjecture is a longstanding conjecture in 3-manifold theory. We present a theorem  about exceptional cosmetic surgery for homology spheres. Along the way we prove that if the surgery is not a small seifert $\Z/2\Z$-homology sphere or a toroidal irreducible non-Seifert surgery then there is at most one pair of exceptional truly cosmetic slope. We also prove that toroidal truly cosmetic surgeries on integer homology spheres must be integer homology spheres.
\end{abstract}
\maketitle

\section{introduction}
In \cite{RavelomananaS3-paper} we proved that for hyperbolic knots in $S^3$ the slope of \textit{exceptional truly cosmetic surgeries} must be $\pm 1$ and that the surgery must be irreducible toroidal and not Seifert fibred. As a consequence we showed that there are no truly cosmetic surgeries on alternating and arborescent knots in $S^3$. Here we study the problem for the case of integer homology spheres in general. Recall that a surgery on a hyperbolic knot is \textit{exceptional} if it is not hyperbolic and that two surgeries on the same knot but with different slopes are called \textit{cosmetic} if they are homeomorphic and are called \textit{truly cosmetic} if the homeomorphism preserves orientation. The cosmetic surgery Conjecture \cite[Conjecture (A) in problem 1.81 ]{Kirby-list} states that if the knot complement is boundary irreducible and irreducible then  two surgeries on inequivalent slopes are never truly cosmetic.
For more details about cosmetic surgeries we refer to \cite{Matignon-hyp-knot},\cite{Weeks}, \cite{Matignon-knot}, \cite{Matignon} and \cite{RavelomananaS3-paper}.

%%%%%%%%%%%%%%%%%%%%%%%%%%%%%%%%%%%%%%%%%%%%%%%%%%%%%%%%%%%%%%%%%%%%%%%%%%%%%%%%%%%%%%%

In this paper we study truly cosmetic surgeries along hyperbolic knots in  homology spheres.  We are concerned with the case where the two slopes are both exceptional slopes. We call such surgeries \textit{exceptional truly cosmetic surgeries}. If $K$ is a knot in an integer homology sphere $Y$, we denote $ \nb( K)$  a regular neighbourhood of $K$,   $Y_K:= Y \setminus \text{int}(\nb(K))$ the exterior of $K$ and $Y_K(r)$  Dehn surgery along $K$ with slope $r$. When the manifold is a homology sphere we identify $r$ with a rational number according to the standard meridian longitude basis where the longitude is the preferred longitude. The main result of this paper is the following.

\begin{thm}\label{my-proposition-2}
Let $K$ be a hyperbolic knot in a homology sphere $Y$. Let $0<p$ and $q<q'$ be  integers.
If $Y_K(p/q)$ is homeomorphic to $Y_K(p/q')$ as oriented manifolds, then $Y_K(p/q)$ is  either
\begin{enumerate}
\item a reducible manifold  in which case $p=1$ and $q'=q+1$,
\item a toroidal Seifert fibred manifold  in which case $p=1$ and $q'=q+1$,
\item a small Seifert  manifold with infinite fundamental group in which case either
\begin{itemize}
\item[•] $p=1$ and $|q-q'|\leq 8$.
\item[•] or $p=5, \ q'=q+1$ and $q \equiv 2 \left[\text{mod}\ 5\right] $.
\item[•] or $p=2, \ \text{and} \ q'=q+2$ or $q'=q+4$.
\end{itemize}
\item a toroidal irreducible non-Seifert fibred manifold  in which case $p=1$ and $|q'-q| \leq 3$.
\end{enumerate}
\end{thm}

The following two corollaries  are straigtforward consequences of the theorem.

\begin{cor}\label{corollary-toroidal-give-integer}
Toroidal truly cosmetic surgeries along hyperbolic knots in integer homology spheres yield integer homology spheres.
\end{cor}

\begin{cor}\label{corollary-at-most-1-pair}
For a hyperbolic knot in an homology sphere there is at most one pair of exceptional truly cosmetic slope which does not yield a $\Z/2\Z$-homology small Seifert surgery or a toroidal irreducible non-Seifert surgery.
\end{cor} 

\paragraph*{Notations.}
If a torus $T$ is a component of $\partial M$ we denote $M(s,T)$ the Dehn filling of $M$ with slope $s$ along $T$, if $\partial M$ has only one torus component we will simply write $M(s)$. In the case of surgery along a knot $K$ is a 3-manifold $Y$ we use the notation $Y_K(s)$ defined earlier.

\textit{Rational longitude.}  Let $K$ be a knot in a rational homology 3-sphere $Y$. The knot $K$ has finite order in $H_1(Y,\Z)$ so there is an integer $n$ and a surface $\Sigma \subset Y$ such that $nK=\partial \Sigma$. The intersection of $\Sigma$ with $\partial \nb (K)$ is $n$-parallel copies of a curve $\lambda_M$. The isotopy class in $\partial \nb (K)$ of this curve does not depend on the choice of the surface $\Sigma$. We call the corresponding slope the rational longitude of $K$ and denote it by  $\lambda_M$. 

We will need the following lemma from \cite{Watson-PhD}.
\begin{lem} \cite{Watson-PhD} \label{watson-lemma}
Let  $s$ be  a slope on $ \partial Y_K$. There is a constant $c_M$ such that
$$|H_1(Y_K(s);\Z)|= c_M\ \Delta(s,\lambda_M).$$
\end{lem}
Here $\Delta(r,s)$ stands for the distance between two slopes $r$ and $s$ i.e their minimal geometric intersection.

\subsection*{Acknowledgment.}
The author would like to thank Steven Boyer for helpful discussions and comments. 
 \section{Proof of Theorem~\ref{my-proposition-2}} \label{Chapter: Exceptional Cosmetic Surgery on Integer Homology Sphere}% renisoratra daholo rehefa manomboka

Let us first recall a result of Boyer and Lines about the second derivative of the Alexander polynomial $\Delta''_K$ of a knot $K$. 

\begin{pro}\cite{Boyer-Lines}\label{my-proposition-1} Let $K$ be a non-trivial knot in a homology 3-sphere $Y$ and let $r$ and $s$ be two distinct slopes. If there is an orientation preserving homeomorphism between $Y_K(r)$ and $Y_K(s)$ then $\Delta''_K(1) = 0$.
\end{pro}
A consequence of this is the following lemma.
\begin{lem}\label{lemma-dedekind}
Let $K$ be a non-trivial knot in a homology 3-sphere $Y$. If  there is an orientation preserving homeomorphism between  $Y_K(p/q)$ and $Y_K(p/q')$ then  $s(q,p)=s(q',p)$, where$$s(q,p):=\text{sign}(p)\cdot \sum^{|p|-1}_{k=1}((\frac{k}{p}))((\frac{kq}{p})),$$
with $$ ((x))=\begin{cases} x-[x]-\frac{1}{2}, & \text{if $x \notin \Z$}, \\
0, & \text{if $x \in \Z$}.  \end{cases}$$
\end{lem}
\begin{proof}
By using the surgery formula\footnote{we have a difference of sign here due to our convention for lens spaces}  \cite[Corollary 4.5]{Saveliev} for the Casson invariant $\lambda$, an invariant for oriented rational homology sphere, we have
$$\lambda(Y)+\lambda(L(p,q))+\frac{q}{2p}\Delta_K''(1)=\lambda(Y_K(p/q))=\lambda(Y_K(p/q'))=\lambda(Y)+\lambda(L(p,q'))+\frac{q'}{2p}\Delta_K''(1)$$
but $\Delta''_K(1) = 0$ by Proposition~\ref{my-proposition-1} so $\lambda(L(p,q))=\lambda(L(p,q'))$. On the other hand Boyer and Lines in \cite{Boyer-Lines} have computed the Casson invariant for a lens space $L(p,q)$ to be
 $$\lambda(L(p,q))=-\frac{1}{2}s(q,p).$$
\end{proof}

Let $Y$ be an integer homology sphere and let $K$ be a knot in $Y$. 
Assume $Y_K(p/q)\cong + Y_K(p/q')$ then we have an  induced isomorphism on homology and between the linking pairing of $Y_K(p/q)$ and $Y_K(p/q')$. More precisely let  $\left[\mu \right]_q $ (resp.   $\left[\mu \right]_{q'} $) be the meridian generator of $H_1(Y_K(p/q))$ (resp. $H_1(Y_K(p/q'))$), then we can find a unit $ u \in \left(\Z/p\Z\right)^*$ such that 
 $$f_*\left[\mu \right]_q =\left[\mu \right]_{q'} u.$$

% A . B --->  - A . B  raha mivadika ilay orientation
On the other hand if we denote $lk_q$ (resp. $lk_{q'}$) the linking pairing of $Y_K(p/q)$ and $Y_K(p/q')$ respectively and $f_*$ the map induced on homology then we can check that   
$$lk_q(\left[\mu \right]_q,\left[\mu \right]_q)=-q/p \ \ \left[\text{mod}\ \Z\right].$$
The isomorphism $f_*$ then gives
$$
\begin{aligned}
lk_q(\left[\mu \right]_q,\left[\mu \right]_q) & =lk_q(f_*\left[\mu \right]_q,f_*\left[\mu \right]_q) \ \ \ \left[\text{mod}\ \Z\right]\\
&=lk_q(\left[\mu \right]_{q'},\left[\mu \right]_{q'}) u^2\ \ \ \left[\text{mod}\ \Z\right]
\end{aligned}$$
and it follows that
\begin{equation}\label{congruence}
-\frac{q}{p} \equiv -\frac{q'}{p} u^2 \ \ \ \left[\text{mod}\ \Z\right], \  \ \text{i.e} \ \ \ q \equiv q' u^2 \ \ \ \left[\text{mod}\ p\right].
\end{equation}
 
As a consequence  we have the following lemma.

\begin{lem}\label{my-lemma-2}
Let $K$ be a hyperbolic knot in a $\Z$-homology sphere $Y$. Let $p/q$ and $p/q'$ be exceptional slopes such that $0<p$ and  $q < q'$. If there is an orientation preserving homeomorphism between $Y_K(p/q)$ and $Y_K(p/q')$ then one of the following holds:

\begin{itemize}
\item[(a)]  $p=1$ and $|q-q'|\leq 8$.
%\item[(b)] $p=7, \ q'=q+1$ and $q \equiv 5 \left[\text{mod}\ 7\right] $ or $q \equiv 1 \left[\text{mod}\ 7\right] $.
\item[(b)] $p=5, \ q'=q+1$ and $q \equiv 2 \left[\text{mod}\ 5\right] $.
\item[(c)] $p=2, \ \text{and} \ q'=q+2$ or $q'=q+4$.
\end{itemize}

\end{lem}

\begin{proof}
Since the slopes are exceptional by \cite[Theorem 1.2]{Lackenby-Meyerhoff} $\Delta(p/q,p/q')=|pq'-qp|=p |q-q'| \leq 8$  so $p \leq 8$. If $p \in \left\{8,7,6,5\right\}$ then $|q-q'|\leq 1$ and  $q'=q+1$. Since one of $q$ and $q+1$ is even and $p,q$ (resp. $p,q'$) are relatively prime  $p$ cannot be $6$ or $8$.  Similarly if $p\in \left\{4,3\right\}$ then  $|q-q'|\leq 2$ and $q'\in \left\{q+1,q+2\right\}$. If $p=2$ then   $|q-q'|\leq 4$ and $q'\in \left\{q+1,q+2,q+3,q+4\right\}$ but  we must have $q\equiv q'$ $[\text{mod}\ 2]$ by Equation~\ref{congruence} so $q'\in \left\{q+2,q+4\right\}$.  

We  use the same equation for $p\in \left\{7,5,4,3\right\}$ to obtain the result.  
\begin{itemize}

\item[•] Case $p=7$. The squares modulo $7$ are $1$, $2$ and $4$, they are all units so
  $$q \equiv q+1 \ \ \ \left[\text{mod}\ 7\right]\ \ \ \text{ or}\ \ \ \  q \equiv 2 (q+1) \ \ \ \left[\text{mod}\ 7\right]\ \ \ \text{ or}\ \ \ \  q \equiv 4 (q+1) \ \ \ \left[\text{mod}\ 7\right].$$
The first equation is impossible and the last two are equivalent to
  $$q \equiv 5 \ \ \ \left[\text{mod}\ 7\right]\ \ \ \text{ or}\ \ \ \  q \equiv 1 \ \ \ \left[\text{mod}\ 7\right].$$
%$7-2=5$ ary 3q=-1 i.e 6q=-2, nefa 7-1=6 --> q=2
 By a straightforward computation
$$s(5,7)=\frac{-1}{14}, \ \ \ s(6,7)=\frac{-5}{14}, \ \ \ s(1,7)=\frac{5}{14}, \ \ \ s(2,7)=\frac{1}{14}.$$ 
Using the fact that $s(a,p)=s(b,p)$ if $a \equiv b \ \ \ \left[\text{mod}\ p\right]$, we get

If $q \equiv 5 \ \ \ \left[\text{mod}\ 7\right]$, 
$$s(q,7)=s(5,7)=\frac{-1}{14}\neq \frac{-5}{14}=s(6,7)=s(q+1,7)$$
If $q \equiv 1 \ \ \ \left[\text{mod}\ 7\right]$, 
$$s(q,7)=s(1,7)=\frac{5}{14}\neq \frac{1}{14}=s(2,7)=s(q+1,7)$$
This contradicts Lemma \ref{lemma-dedekind}  which says that we must have $s(q,p)=s(q',p)$.
\item[•] Case $p=5$. The squares modulo $5$ are $1$ and $4$, the only unit among them is $1$, therefore 
  $$q \equiv q+1 \ \ \ \left[\text{mod}\ 5\right]\ \ \ \text{ or}\ \ \ \  q \equiv 4(q+1) \ \ \ \left[\text{mod}\ 5\right].$$
The first equation has no solution and the second is equivalent to 
$$q \equiv 2 \ \ \ \left[\text{mod}\ 5\right].$$

% s(2,5)=s(3,5)=0 eto dia tsy afaka manao n'inoinona.

\item[•] Case $p=4$.  We have   $q'\in \{q+1,q+2\}$ and the only square modulo $4$ is $1$ therefore
  $$q \equiv q+1 \ \ \ \left[\text{mod}\ 4\right]\ \ \ \text{ or}\ \ \ \  q \equiv q+ 2 \ \ \ \left[\text{mod}\ 4\right].$$
These equations have no solutions so the case $p=4$ is not possible.

\item[•] Case $p=3$.   We have $q'\in \{q+1,q+2\}$ and the only square modulo $3$ is $1$, therefore this case is also impossible.

\end{itemize}
\end{proof}

The next two theorems of Gordon and Wu about toroidal surgery from \cite{toroidal} and \cite{punctured-tori} will play a key part in the proof of Theorem~\ref{my-proposition-2}. The first theorem is about pairs of toroidal slopes at distance $4$ or $5$ apart and the second theorem is for distance greater than $5$. 
\begin{thm}\cite{toroidal}\label{toroidal:Theorem 1.1}
There exist fourteen hyperbolic $3$-manifolds $M_i$, $1 \leq i \leq 14$, such that
\begin{enumerate}
\item $\partial M_i$ consists of two tori $T_0, T_1$ if $i \in
\{1,2,3,14\}$, and a single torus $T_0$ otherwise;

\item there are slopes $r_i, s_i$ on  $T_0$ such that $M_i(r_i,T_0)$ and $M_i(s_i,T_0)$ are toroidal,
\begin{itemize}
\item $\Delta(r_i,s_i) = 4$ if $i \in \{1,2,4,6,9,13,14\}$, and
\item $\Delta(r_i,s_i) = 5$ if $i \in \{3,5,7,8,10,11,12\}$;
\end{itemize}

\item if $M$ is a hyperbolic 3-manifold with toroidal Dehn fillings
$M(r), M(s)$ where $\Delta(r,s) = 4$ or $5$, then $(M,r,s)$ is
equivalent either to $(M_i,r_i,s_i)$ for some $1\leq i \leq 14$, or to
$(M_i(t,T_1),r_i,s_i)$ where $i \in \{1,2,3,14\}$ and $t$ is a slope on
the boundary component $T_1$ of $M_i$.

\end{enumerate}
Here we define two triples $(N_1, r_1, s_1)$ and $(N_2, r_2,
s_2)$ to be {\it equivalent} if there is a homeomorphism from $N_1$ to $N_2$ which
sends the boundary slopes $(r_1, s_1)$ to $(r_2, s_2)$ or $(s_2,
r_2)$.
\end{thm}

 Let $W$ be the exterior of the Whitehead link and let $T_0$ be a boundary component of $W$. Choosing a standard meridian-longitude basis $\mu,\lambda$ for $H_1(T_0)$ we  can identify  slopes $T_0$ with elements of $\Q\cup \{1/0\}$.  
The manifolds  $W(1),\ W(2),\ W(-5),$ $ W(-5/2)$  are hyperbolic and they all admits a pair of toroidal slopes $r,s$ with $\Delta(r,s)>5$. Gordon proved that these examples are the only possibilities for hyperbolic manifolds with  pair of toroidal slopes at distance $>5$.

\begin{thm}\cite{punctured-tori}\label{theorem-Gordon-punctured-tori}
Let $M$ be an irreducible 3-manifold and $T$ a torus component of $\partial M$. If two slopes $r$ and $s$ on $T$ are toroidal then either
\begin{enumerate}
\item $\Delta(r,s)\leq 5$; or
\item $\Delta(r,s) =6 $ and $M$ is homeomorphic to $W(2)$; or
\item $\Delta(r,s) =7 $ and $M$ is homeomorphic to $W(-5/2)$; or
\item $\Delta(r,s) =8 $ and $M$ is homeomorphic to $W(1)$ or $W(-5)$.
\end{enumerate}
\end{thm}

We can now get more restrictions on  the slopes which gives cosmetic toroidal fillings.
\begin{lem}\label{large-distance-toroidal}
Let $K$ be a knot in a $\Z$-homology sphere $Y$. Let  $r,s$ be two slopes on $\partial Y_K$. If $Y_K(r)$ and $Y_K(s)$ are toroidal and if there is an orientation preserving homeomorphism between them, then $\Delta(r,s) \leq 3$.
\end{lem}

\begin{proof}

We will distinguish the cases $\Delta(r,s) > 5$  and $\Delta(r,s) = 5$, or $4$. 

Let $W$ be the Whitehead link exterior. By Theorem \ref{theorem-Gordon-punctured-tori}   if  $\Delta(r,s) > 5$ then  either
\begin{itemize}
\item[•] $\Delta(r,s)=6$ and $Y_K$ is homeomorphic to $W(2)$
\item[•] $\Delta(r,s)=7$ and $Y_K$ is homeomorphic to $W(-5/2)$
\item[•] $\Delta(r,s)=8$ and $Y_K$ is homeomorphic to $W(1)$ or $W(-5)$
\end{itemize}
The manifold $Y_K(r)$ is then obtained by surgery on the Whitehead link with coefficients $\{2,a_1/b_1\}$ or $\{-5/2,a_2/b_2\}$ or $\{1,a_3/b_3\}$ or  $\{-5,a_4/b_4\}$.  We can then compute the order of the first homology using this coefficient, 
$$
|H_1(Y_K(r))|=\begin{vmatrix}
2& b_1\  \text{lk}(K_1,K_2)\\
\text{lk}(K_2,K_1)& a_1
\end{vmatrix}=\begin{vmatrix}
2& 0\\
0& a_1
\end{vmatrix}
=2a_1
$$
where $K_1,K_2$ denotes the two components of the Whitehead link and $\text{lk}(K_1,K_2)$ their linking number. Similarly we get for the other possibilities
$$
|H_1(Y_K(r))|=\begin{vmatrix}
-5& 0\\
2& a_2
\end{vmatrix}
=-5a_2, \ \ \ \text{or} \ \ \ 
|H_1(Y_K(r))|=\begin{vmatrix}
-5& 0\\
0& a_4
\end{vmatrix}
=-5a_4.
$$

On the other hand if $\Delta(r,s) > 5$  then $Y_K(r)$ must be a homology sphere by Lemma \ref{my-lemma-2}. Therefore these three possibilities cannot occur. The remaining case is $Y_K(r)=W(1)$ which is the figure-8 complement and there are no truly cosmetic surgeries along the figure-8 knot, as we can check for instance that $\Delta_{\text{figure-8}}''(1)\neq 0$ and use Proposition~\ref{my-proposition-1}.

Now we can assume $\Delta(r,s) \in \{4, 5\}$. We will do a case by case study using Theorem \ref{toroidal:Theorem 1.1}.

\begin{itemize}
\item[•] Case 1. $Y_K$ is one of $M_1,M_2,M_3$. 

The manifolds  $M_1,M_2,M_3$ are the exterior of the following links \cite{toroidal}

\begin{figure}[h!]
\centerline{\epsfbox{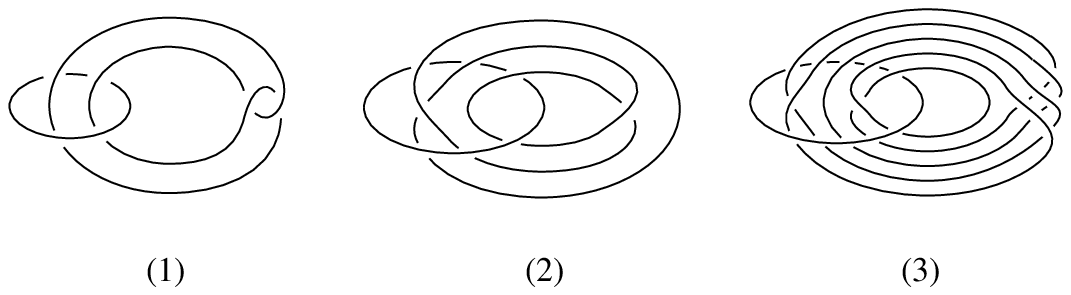}}
\end{figure}

For $i\in \{1,2,3\}$ we will denote by $K'_i$ and $K''_i$ the leftmost and rightmost components of the above links. From \cite{toroidal} we know that the component $T_1$ is the boundary of $\nb(K'_i)$.
Let $t=a/b$ be a slope on $\nb(K'_i)$ written in its Seifert framing. If $a\neq 1$ then $H_1\left(M_i(t)\right)=\Z \oplus \Z/a\Z \neq \Z = H_1(Y_K)$. Therefore we must have $a=1$. But in this case we have a knot complement in $S^3$ and we know from \cite[Theorem 1.4]{RavelomananaS3-paper} that the surgery must be $\pm 1$ therefore $\Delta(r,s)=2$ which is not the case here.

\item[•] Case 2. $Y_K \cong M_{14}$. Let $t$ be a slope on the boundary component $T_1$   of $M_{14}$, and let $K_t$ be the core of the Dehn filling solid torus in $M_{14}(t,T_1)$. From \cite[Theorem 22.3]{toroidal} we can compute $$H_1\left(M_{14}(t,T_1)\right)/H_1(K_t)= \Z/2\Z \oplus \Z/2\Z.$$ Therefore $H_1\left(M_{14}(t)\right)\neq \Z = H_1(Y_K)$.

\item[•] Case 3.  $Y_K$ is one of $M_4$ and $M_5$. From \cite[Theorem 22.3]{toroidal} we can also compute $$H_1(M_4(r))=\Z, \ \ \text{ and }\ \  H_1(M_5(r))=\Z \oplus \Z/4\Z.$$
These situations are not possible since $Y_K(r)$ is a rational homology sphere.

\item[•] Case 4. $Y_K$ is one of $M_6,M_7,M_{10},M_{11},M_{12}, M_{13}$. By \cite[Lemma 22.2]{toroidal} $Y_K$ admits a Lens space surgery. With respect to the framing in \cite{toroidal} these lens space surgeries  are
$$
  \begin{array}{lll}
   M_6(\infty)  = L(9,2) \quad  &   M_7(\infty)  = L(20,9) \quad  & M_{10}(\infty)  = L(14,3) \\
  M_{11}(\infty)  = L(24,5)  \quad  & M_{12}(\infty)  = L(3,1)  \quad &    M_{13}(\infty) = L(4,1)
\end{array}
$$
% Mampiasa ny lemme an'i Liam
From this we can deduce that $|\text{Tor}\left(H_1(Y_K)\right)| \neq 1$ which is not possible since $H_1(Y_K) =\Z$.%since $\infty$ represents the meridian.
\item[•] Case 5. $Y_K$ is one of $M_8$ and $M_9$. From \cite[Lemma 22.2]{toroidal} the manifolds $M_8$ and $M_9$ has two toroidal surgeries and one lens space surgery listed as follows with respect to the framing used in \cite{toroidal}
$$
  \begin{array}{lll} 
    M_8(0), \quad & M_8(-5/4), 
    \quad  & M_8(-1)  = L(4,1) \\
    M_9(0), \quad & M_9(-4/3),  \quad  & M_9(-1)  = L(8,3) 
\end{array}
$$
For $i \in \{8,9\}$ let $a=|\text{Tor}\left(H_1(M_i)\right)|$ and $l$ be the order of the preferred rational longitude $\lambda_{M_i}$. We are going to express the framing used in \cite{toroidal} according to our standard basis $\{\mu, \lambda_{M_i}\}$. Let $\lambda$  be the framing used in \cite{toroidal}. Then the $-1$ slope in this framing can be written $-\mu + (p\mu+\lambda_{M_i})=(p-1)\mu +\lambda_{M_i}$. 
Using the fact that $|H_1\left(L(4,1)\right)|=4$ and $|H_1\left(L(8,3)\right)|=4$,  with Lemma~\ref{watson-lemma} we get
$$|H_1\left(M_8(-1)\right)|=4=\Delta\left((p-1)\mu +\lambda_{M_8}\ ;\ \lambda_{M_8}\right)\ la=|p-1| la,$$
$$|H_1\left(M_9(-1)\right)|=8=\Delta\left((p-1)\mu +\lambda_{M_9}\ ;\  \lambda_{M_9}\right)\  la=|p-1| la.$$
Since $Y_K$ is a knot complement in an integer homology sphere, if $Y_K$ is one of $M_8$ or $M_9$ then we must have $l=a=1$. Therefore $p\in \{-3,5\}$ for $M_8$ and $p\in \{9,-7\}$ for $M_9$. We can then deduce
% -5/4 =-5 \mu + 4 (5 or -3) \mu + 4 \lambda_M = (-5+20 or -5 -12 ) \mu + 4 \lambda_M
% -4/3 =-4 \mu + 3 (9 or -7) \mu + 3 \lambda_M = (-4+27 or -4 -21 ) \mu + 3 \lambda_M
$$
  \begin{array}{ll} 
   H_1(M_8(0))= \Z/5\Z \ \ \text{or} \  \Z/3\Z, \quad & H_1(M_8(-5/4))=  \Z/15\Z \ \ \text{or} \ \Z/17\Z, \\
   H_1(M_9(0))=\Z/9\Z\ \ \text{or} \  \Z/7\Z, \quad & H_1(M_9(-4/3))=\Z/23\Z  \ \ \text{or} \ \Z/25\Z, 
\end{array}
$$
Therefore $M_8(0)$ and $M_8(-5/4)$ are not homeomorphic and the same is true for $M_9(0)$ and $M_8(-4/3)$.  We can conclude that $Y_K$ cannot be one of $M_8$ or $M_9$.
\end{itemize}
\end{proof}

The last lemma implies in particular that toroidal truly cosmetic surgeries on integer homology sphere must be integer homology spheres.

%%%%%%%%%%%%%%% Nampiana  16 Naovambra 2014 %%%%%%%%%%%%%%%%%%%%%%%%%%%%%%%%%%%%%%%%%%

%The second situation arises under the assumption of a certain type of Seifert filling of M .
The next preliminary result addresses the case of Seifert fibred toroidal surgeries. 
Before going into it we need a bit of $PSL_2(\C)$-character variety theory. We refer to \cite{Boyer-Zhang} for more details about character variety theory.%  and Chapter \ref{Chapter:Cosmetic surgery and character varieties} for more details on the subject.

Let $X(G)$ denote the $PSL_2(\C)$-character variety of a finitely generated group $G$. When $G=\pi_1(Z)$ where $Z$ is a path-connected space, we shall write $X(Z)$ for $X(\pi_1(Z))$.  Recall that $X(G)$ is a complex algebraic variety and a surjective homomorphism $G\twoheadrightarrow H$ induces an injective morphism $X(H) \hookrightarrow X(G)$ by precomposition. A curve $X_0 \subset X(G)$ is called non-trivial if it contains
the character of an irreducible representation. Each  $\gamma \in X(G)$ determines an element $f_{\gamma}$ of the coordinate ring $\C[X(G)]$ where if $\rho: G \to PSL_2(\C)$ is a representation and $\chi_{\rho}$ the associated point in $X(G)$, then 
$f_{\gamma}(\chi_{\rho})=\text{trace}(\rho(\gamma))^2-4$.
%(see eg. [BZ1, §3]). 
When $G=\pi_1(M)$, any slope $r$ on $\partial M$  determines  an element of
$\pi_1(M)$,  well-defined up to conjugation and taking inverse. Hence it induces a well-defined element $f_r \in \C[X(M)]$.

%%%%%%%%%%%%%%%%%%%%%%%%%%%%%%%%%%%%%%%%%%%%%%%%%%%%%%%%%%%%%%%%%%%%%%%%%%%%%%%%%%%%%%

\begin{lem}\label{my-lemma-3}
Let $K$ be a hyperbolic knot in an integer homology sphere $Y$. Let $r=p/q$ and $r'=p/q'$ be exceptional slopes such that $0<p$ and  $q < q'$.  If $Y_K(r)$  is homeomorphic to $Y_K(r')$ as oriented manifolds and is Seifert fibred and toroidal, then $p=1$ and $q'=q+1$.
\end{lem}
\begin{proof}

Let  $\mathcal{B}$ be  the base orbifold for $Y_K(r)$. Since $Y_K(r)$ is toroidal with finite first homology $\mathcal{B}$ cannot be  spherical. Moreover it cannot be a sphere with strictly less than $4$ cone points. Thus $\mathcal{B}$ must be either  hyperbolic  or one among: $S^2(2,2,2,2)$, $\mathbb{T}^2$, $\RP^2(2,2)$ or the Klein bottle.
%%% Nocommentena ny 7 Aprily 2015 taorian ny fametrahana ary ny corrections nomen i Steven.%%%%%%%%%%%%%%%%%%%%%%%%%%%%%%%%%%%%%%%%%%%%%%%%%%%%%%%%%%%%%%%%%%%%%%%%%%%
Since we assume that $Y_K(r)$ and $Y_K(r')$ are toroidal, by lemma~\ref{large-distance-toroidal} $\Delta(r,r') \leq 3$, so $p\leq 3$. Lemma \ref{my-lemma-2} then implies that  $p=1$ or $p=2$. If $p=2$ then $q'=q+2$ or $q'=q+4$ and it follows that $\Delta(r,r')=4$ or $8$, but this contradicts the fact that $\Delta(r,r') \leq 3$. Therefore we must have $p=1$.
Furthermore using the fact that $Y_K(r)$ is a Seifert fibred manifold, we have the following surjection in first homology: 
$H_1(Y_K(r)) \twoheadrightarrow H_1(\mathcal{B})$,
 thus $|H_1(Y_K(r))|=p= 1 \geq  |H_1(\mathcal{B})|$. However we know that $|H_1(S^2(2,2,2,2))|=|H_1(\RP^2(2,2))|=8$, $H_1(\mathbb{T}^2)=\Z\oplus \Z$ and $H_1(\text{Klein bottle})=\Z\oplus\Z/2\Z$. Thus   $\mathcal{B}$ must be  hyperbolic. 
 
By the same argument as above $\mathcal{B}$ cannot be $\RP^2(a,b)$ since   $|H_1(\RP^2(a,b))|=2ab > 1$.

%%%%%%%%%%%%%%%%%%%%%%%%%%%%%%%%%%%%%%%%%%%%%%%%%%%%%%%%%%%%%%%%%%%%%%%%%%%%%%%%%%%%%%

%%%%%%%%%%%%% Nampiana 17 naovambra 2014 %%%%%%%%%%%%%%%%%%%%%%%%%%%%%%%%%%%%%%%%%%%%%%%

By work of Thurston \cite{Thurston}, since $\mathcal{B} \neq \RP^2(a,b)$  the real dimension of the Teichm\"uller space $\mathcal{T}(\mathcal{B})$ of $\mathcal{B}$ is at least $2$. Moreover $\mathcal{T}\subset X(\pi_1^{orb}(\mathcal{B}))$ where $\pi_1^{orb}(\mathcal{B})$ is the orbifold fundamental group of $\mathcal{B}$. On the other hand we have
$$\pi_1(M)\twoheadrightarrow \pi_1(M(r)) \twoheadrightarrow \pi_1^{orb}(\mathcal{B})$$ 
which induces a sequence of inclusions 
$$X(Y_K) \supset X(Y_K(r)) \supset X(\pi_1^{orb}(\mathcal{B})) \supset \mathcal{T}(\mathcal{B}).$$
Therefore the complex dimension of $X(Y_K(r))$ is at least $1$. We want to prove that it contains a subvariety of complex dimension at least $2$.  Assume on the contrary that all components of $X(Y_K(r))$ have  complex dimension $1$. In this case $\mathcal{T}(\mathcal{B})$ 
 would be an open set in a non-trivial curve $X_0 \subset X(Y_K(r))$. When $\chi_{\rho} \in \mathcal{T}(\mathcal{B})$, $\rho$ is the holonomy of a hyperbolic structure on $\mathcal{B}$ and it is well known that if $\gamma \in \pi_1^{orb}(\mathcal{B})$ has infinite order, then $f_{\gamma}(\chi_{\rho})$ is a real number.
 % which is essentially the length of the unique geodesic in this structure representing the conjugacy class of γ (see eg. [FLP, Lemme 1, page 135]).
  Deforming $\chi_{\rho}$ in $\mathcal{T}(\mathcal{B})$ shows that $f_{\gamma}|_{X_0}$
is non-constant and  must take some non-real values. This contradicts the fact that it is
real-valued on the open subset  $\mathcal{T}(\mathcal{B})\subset X_0$.
Thus $X(Y_K)$ has a subvariety of complex dimension $2$ or larger on which $f_r$ is constant and which contains the character of an irreducible representation. Hence if $r'\neq r$  is any
other slope, we can then construct a non-trivial curve $X_0 \subset X(Y_K)$ on which both $f_r$ and $f_{r'}$ are constant. Indeed let $X$ be this two dimensional subvariety, if $f_{r'}|_{X}$ is constant then we are done, otherwise we can take a regular value $z_0\in \C$ of $f_{r'}|_{X}$, the preimage $f_{r'}|_{X}^{-1}(z_0)$  is a codimension one subvariety of $X$ and we can take $X_0=f_{r'}|_{X}^{-1}(z_0)$. It follows that $f_r|_{X_0}$  is constant for each slope. In particular  for each ideal point $\tilde{x}$ of $X_0$ and slope $s\in \partial Y_K$, $\tilde{f}_s(\tilde{x})\in \C$. Now \cite[Proposition 4.10 \& Claim of page 786]{Boyer-Zhang}  imply that there is a closed essential surface $S\subset  Y_K$  which compresses in $Y_K(r)$ but stays incompressible in $Y_K(s)$ if $\Delta(s,r) >1$. 
%%%%%%%%%%%%%%%%%%%%%%%%%%%%%%%%%%%%%%%%%%%%%%%%%%%%%%%%%%%%%%%%%%%%%%%%%%%%%%%%%%%%%%%%

Suppose we have $\Delta(r,r') \geq 2 $, then $S$ must be incompressible in $Y_K(r')$. Since $Y_K$ is hyperbolic it has no incompressible torus. Therefore $S$ must have genus at least $2$ and is a horizontal surface. 

On the other hand $Y_K\subset Y$ and $Y$ is a $\Z$-homology sphere so $S$ must separate $Y_K$ and also $Y_K(r')$. Indeed $H_2\left(Y\right)=0$ so $\left[S\right]=0$ and  $S$  separates.
%is a boundary of some 3-manifold. Since $S$ is also two-sided it bounds two distinct 3-manifold in both side. 
 Let $M_1$ and $M_2$ be the two components of $Y_K(r') \setminus S$. They are both interval semi-bundles with base $\mathcal{B}$. It follows that if $\Sigma_i$ is the core surface of $M_i$, then $\pi_1(\Sigma_i) \cong \pi_1(M_i) $ for $i=1,2$. On the other hand since $\partial M_i =S$ is connected, we have a $2$ to $1$ connected cover $\partial M_i \to \Sigma_i$. Then $\pi_1(S)$ is an index two subgroup of  $\pi_1(\Sigma_i)$, in particular it is normal. Using  Van-Kampen theorem we have
$$\pi_1(Y_K(r'))=\pi_1(\Sigma_1) \ast_{\pi_1(S)} \pi_1(\Sigma_2)$$
 and $\pi_1(S)$ is normal in $\pi_1(Y_K(r'))$ since it is normal in both component of the amalgam. Hence
 $$\frac{\pi_1(Y_K(r'))}{\pi_1(S)}=\left(\frac{\pi_1(\Sigma_1)}{\pi_1(S)}\right) \ast \left(\frac{\pi_1(\Sigma_2)}{\pi_1(S)}\right)\cong \Z/2Z \ast \Z/2\Z$$
and we have a surjection $\pi_1(Y_K(r')) \twoheadrightarrow \Z/2\Z \ast \Z/2\Z$. This induces a surjection in first homology $H_1(Y_K(r')) \twoheadrightarrow \Z/2\Z \oplus \Z/2\Z$, which contradicts the fact that $H_1(Y_K(r'))$ is cyclic.
Therefore $\Delta(r,r')=p|q-q'| \leq 1$ which implies $p=1$ and $q'=q+1$.
\end{proof}

%%%%%%%%%%%%%%%%%%%%%%%%%%%%%%%%%%%%%%%%%%%%%%%%%%%%%%%%%%%%%%%%%%%%%%%%%%%%%%%%%%%%%%%%%

%%%%%%%%%%%%%%%%%%%%%%%%%%%%%%%%%%%%%%%%%%%%%%%%%%%%%%%%%%%%%%%%%%%%%%%%%%%%%%%%%%%%%%%%%

We can now prove the main theorem.

\begin{thm}
Let $K$ be a hyperbolic knot in a homology sphere $Y$. Let $0<p$ and $q<q'$ be  integers.
If $Y_K(p/q)$ is homeomorphic to $Y_K(p/q')$ as oriented manifolds, then $Y_K(p/q)$ is  either
\begin{enumerate}
\item a reducible manifold  in which case $p=1$ and $q'=q+1$,
\item a toroidal Seifert fibred manifold  in which case $p=1$ and $q'=q+1$,
\item a small Seifert  manifold with infinite fundamental group in which case either
\begin{itemize}
\item[•] $p=1$ and $|q-q'|\leq 8$.
\item[•] or $p=5, \ q'=q+1$ and $q \equiv 2 \left[\text{mod}\ 5\right] $.
\item[•] or $p=2, \ \text{and} \ q'=q+2$ or $q'=q+4$.
\end{itemize}
\item a toroidal irreducible non-Seifert fibred manifold  in which case $p=1$ and $|q'-q| \leq 3$.
\end{enumerate}
\end{thm}

\begin{proof} % MBOLA TSY AMPY, TSY VITA
The manifold $Y_K(r)$ is either reducible, Seifert fibred or toroidal. If it is reducible then by \cite[Theorem 1.2]{Gordon-Luecke-3} $\Delta(p/q,p/q')= p|q-q'|=1$. If it is toroidal and Seifert fibred then we have (1) which is given by Lemma \ref{my-lemma-3}. The remaining possibilities are then (3), (4) and the case $\pi_1(Y_K(r))$ is finite. The proofs of (3) and (4) follow from  Lemma~\ref{my-lemma-2}.  We are now left with the last possibility. Assume that $\pi_1(Y_K(r))$ is finite. By \cite[Theorem 1.1]{finite-paper} the distance between two finite slopes is at most $3$, so $\Delta(p/q,p/q')=p|q'-q| \leq 3$. In particular $p\in \{1,2,3\}$, but by Lemma \ref{my-lemma-2}, $p\in \{1,2,5\}$ thus $p=1$ or $p=2$. If $p=2$ then $|q'-q| \geq 2$ by  Lemma \ref{my-lemma-2} and $\Delta(p/q,p/q')=4 > 3$  therefore we can only have $p=1$. It follows that $Y_K(r)$ is a homology sphere with finite fundamental group which implies $Y_K(r)=\Sigma\left(2,3,5\right)$ or $Y_K(r)=S^3$.
If $Y_K(r)=\Sigma\left(2,3,5\right)$ or $S^3$ then $Y_K\subset \Sigma\left(2,3,5\right)$ or $S^3$. Let $Z$ denote either $\Sigma\left(2,3,5\right)$ or $S^3$. Then $Y_K=Z \setminus \nb(K')$ where $K'$ is a non-trivial  knot in $Z$ %(since $M$ is hyperbolic) 
for which there is a non trivial slope which gives $\Sigma\left(2,3,5\right)$. We notice that both $\Sigma\left(2,3,5\right)$ and $S^3$ are L-space homology spheres so by \cite[Lemma 3.3]{Ravelomanana}  $\Delta_K''(1)= 2\neq 0$. Therefore by Proposition~\ref{my-proposition-1} there is no orientation preserving homeomorphism between  $Y_K(r)$ and $Y_K(r')$.

\end{proof}

%%%%%%%%%%%%%%%%%%%%%%%%%%%%%%%%%%%%%%%%%%%%%%%%%%%%%%%%%%%%%%%%%%%%%%%%%%%%%%%%%%%%%%%%%%
%\bibliographystyle{plain}
\bibliographystyle{gtart}
\bibliography{exceptional-cosmetic-surgery-on-homology-spheres}
\thanks{University of Georgia, Department of Mathematics, 606 Boyd GSRC, Athens GA, USA.\\  Email: huygens@cirget.ca}

\end{document}